\documentclass[11pt]{article}

\usepackage{tonas-papers}

\newcommand{\NN}{\mathcal{N}}
\newcommand{\CC}{\mathcal{C}}
\newcommand{\DD}{\mathcal{D}}
\newcommand{\CF}{\mathcal{C}_F}

\graphicspath{{./images/}}

\begin{document}

\makeboxedtitle{No Countable Basis for Borel Directed Graphs of Dichromatic Number at Least Three}{
    Tonatiuh Matos-Wiederhold
}{}{
    I prove that the Borel directed graphs whose vertex set admits a partition into two Borel acyclic sets form a $\mathbf\Sigma^1_2$-complete set; equivalently, that deciding whether a Borel directed graph has Borel dichromatic number at least~$3$ is a $\mathbf\Pi^1_2$-complete problem.
    It follows that no countable family of Borel directed graphs can serve as a basis for this class under Borel homomorphism and, more generally, that any basis must be at least as complex as~$\mathbf\Pi^1_2$.

    The proof lifts the classical NP-completeness reduction of Bokal, Fijavž, Juvan, Kayll, and Mohar to the Borel setting, using the coding framework of Thornton.
    Combined with a straightforward reduction from undirected to directed coloring problems, this completes the picture for finite Borel chromatic and dichromatic thresholds: for every finite $k$, the set of Borel (directed) graphs admitting a Borel $k$-(di)coloring is $\mathbf\Sigma^1_2$-complete, and in particular admits no countable basis.
    This contrasts with the uncountable threshold, where a single-element basis exists for Borel chromatic number (Kechris--Solecki--Todorčević) and a continuum-size basis exists for Borel dichromatic number (Raghavan--Xiao).
}

\section{Introduction}\label{sec:borel-dichrom-intro}

The \emph{dichromatic number} of a directed graph was introduced by Neumann-Lara~\cite{NeumannLara1982} as a natural directed analogue of the chromatic number.
A \emph{dicoloring} of a directed graph~$D$ is a map $c\colon V(D)\to k$ such that every color class induces a subgraph without any directed cycles; the \emph{dichromatic number} $\vec\chi(D)$ is the least such~$k$.
Equivalently, $\vec\chi(D)$ is the minimum number of acyclic sets into which the vertex set can be partitioned.

This notion has attracted considerable attention in finite combinatorics.
Neumann-Lara~\cite{NeumannLara1982} established several foundational results, including analogues of classical lower bounds for the chromatic number.
Bokal, Fijavž, Juvan, Kayll, and Mohar~\cite{Bokal_2004_CircularChrom} later introduced the circular dichromatic number of a directed graph and proved that the decision problem of whether a directed graph $D$ satisfies $\vec\chi(D)\leq 2$ or not is NP-complete; this is in contrast with the usual chromatic number for (undirected) graphs, where one gets NP-completeness only after posing the question for three colors or more.

\subsection{Borel chromatic and dichromatic numbers}

In the Borel setting, we consider graphs and directed graphs whose vertex sets are standard Borel spaces and whose edge (or arc) relations are Borel, and we require the colorings to be Borel measurable.
The \emph{Borel chromatic number} $\chi_B(G)$ of a Borel graph~$G$ was introduced by Kechris, Solecki, and Todorčević~\cite{kechris1999borel}, who proved the celebrated \emph{$G_0$-dichotomy}: there exists a single Borel graph~$\mathbb{G}_0$ such that
\begin{align}\label{eq:G0-dichotomy}
    \chi_B(G) > \aleph_0\quad
    \Longleftrightarrow\quad
    \mathbb{G}_0 \leq_B G,
\end{align}
where $\leq_B$ denotes the existence of a Borel homomorphism $\mathbb{G}_0\to G$.
In other words, the graph~$\mathbb{G}_0$ forms a one-element \emph{basis} for the class of Borel graphs with uncountable Borel chromatic number.

This dichotomy breaks down when one moves from uncountable to merely infinite Borel chromatic number.
In \cite{Todorcevic_2021_ComplexProbBorel}, Todorčević and Vidnyánszky recently proved that the set of Borel graphs with infinite Borel chromatic number (equivalently, Borel chromatic number $\geq 4$ for closed subgraphs of the shift graph) is $\mathbf\Sigma^1_2$-complete.
In particular, no countable family of Borel graphs can serve as a basis for this class.
Moreover, their argument rules out the existence of a simple $G_0$-type dichotomy substituting $\aleph_0$ in \eqref{eq:G0-dichotomy} for any value in $3,4,5\dots$.
Interestingly, when substituting the value $2$, there again is a single Borel graph that makes a dichotomy like \eqref{eq:G0-dichotomy} hold.
This is known as the $L_0$~dichotomy and is the main result of \cite{L0paper}.

Very recently, Raghavan and Xiao~\cite{RaghavanXiao2024} established a dichotomy for the \emph{Borel dichromatic number} that generalizes the $\mathbb{G}_0$-dichotomy to directed graphs: they showed that a continuum-size family of directed graphs characterizes when a Borel directed graph has uncountable Borel dichromatic number.
This continuum-size basis consists of pairwise Borel-incomparable directed graphs, but it is unknown whether any smaller basis exists; in particular, the question of whether a \emph{countable} basis suffices is open.

One can interpret the existence of a small (i.e.\ singleton, or finite, or countable) basis as a sort of measure of complexity of the associated decision problem.
For instance, the existence of $\mathbb{G}_0$ says that deciding whether a Borel graph has uncountable Borel chromatic number is ``straightforward:'' one only has to verify one possible ill-behaved property, namely containing a homomorphic copy of $\mathbb{G}_0$, to give a positive answer; and in some sense, finding this homomorphic copy is the \emph{only} way of proving that a Borel graph has uncountable Borel chromatic number.
In contrast, deciding whether a Borel graph has infinite Borel chromatic number is a much more complex task, as no easily describable method exists for deciding this in terms of graph homomorphisms, and so different graphs probably require wildly different proving methods, if any.

The Borel dichromatic number can recover the Borel chromatic number in the following sense.

\begin{remark}\label{rmk:dichrom-recovers-chrom}
	Given an undirected graph $G$, let $\vec{G}$ denote the directed graph obtained by replacing each edge $\{u,v\}$ with two opposing arcs $(u,v)$ and $(v,u)$.
	Then $\vec\chi_B(\vec{G}) = \chi_B(G)$.
	Indeed, a subset of vertices is acyclic in $\vec{G}$ if and only if it is independent in~$G$: any directed cycle in $\vec{G}$ would in particular contain an arc $(u,v)$ with $\{u,v\}\in E(G)$, so a set containing both $u$ and $v$ is neither acyclic in $\vec{G}$ nor independent in~$G$; conversely, an independent set in~$G$ induces a directed graph in~$\vec{G}$ with no arcs at all, hence trivially acyclic.
	Moreover, the map $\mathrm{code}(G)\mapsto\mathrm{code}(\vec{G})$ is $\Delta^1_1$, since it is built from the edge set using products and unions (Lemma~\ref{lemma:toolkit}(1),(3)).
	It follows that the $\mathbf\Sigma^1_2$-completeness of Borel $k$-coloring of undirected graphs~\cite{Todorcevic_2021_ComplexProbBorel} implies the $\mathbf\Sigma^1_2$-completeness of Borel $k$-dicoloring of directed graphs, for every $k\geq 4$.
	The case $k=2$ is not covered by this reduction (since deciding 2-colorability of undirected graphs is trivial), and is precisely the content of Proposition~\ref{prop:main}.
\end{remark}

It is natural to ask whether the Todorčević--Vidnyánszky phenomenon also occurs in the directed setting at the remaining finite threshold, or if there is a directed analog for the $L_0$ dichotomy:

\begin{quote}
\emph{Is there a countable basis for the class of Borel directed graphs with Borel dichromatic number at least~$3$?}
\end{quote}

This paper gives a negative answer, thereby completing the picture for all finite thresholds of both the Borel chromatic and dichromatic numbers.

\subsection{Main results}

I work throughout with \emph{nice codes} in the sense of Thornton~\cite{thornton2022algebraic, thornton-thesis}; see Section~\ref{sec:framework} for a brief review.

\begin{proposition}\label{prop:main}
	The set of nice codes for Borel directed graphs admitting a Borel $2$-dicoloring is $\mathbf\Sigma^1_2$-complete.
	Equivalently, the set of nice codes for Borel directed graphs with Borel dichromatic number at least~$3$ is $\mathbf\Pi^1_2$-complete.
\end{proposition}

\begin{corollary}\label{cor:no-basis}
	There is no countable family~$\mathbf{B}$ of Borel directed graphs that forms a basis for Borel directed graphs of Borel dichromatic number at least~$3$ under Borel homomorphism.
\end{corollary}

The proof of Proposition~\ref{prop:main} proceeds by adapting the classical NP-completeness reduction from 2-coloring of 3-uniform hypergraphs to 2-dicoloring of directed graphs~\cite{Bokal_2004_CircularChrom} to the Borel setting.
The key step is verifying that this reduction is $\Delta^1_1$ in the codes, following the template established by Thornton in~\cite[Appendix~A]{thornton-thesis} for a similar result about edge 3-coloring.
I remark that 2-dicoloring of directed graphs is \emph{not} a CSP in the sense of Thornton's algebraic framework~\cite{thornton2022algebraic}, since acyclicity is not a constraint expressible as a homomorphism to a fixed finite template; thus Proposition~\ref{prop:main} does not follow directly from Thornton's general results.

\subsection{Context and discussion}
 
The following table summarizes the basis landscape, contrasting the undirected and directed cases at different thresholds.

\medskip
\begin{center}
\begin{tabular}{lll}
	\hline
	\textbf{Class} & \textbf{Existence of small basis} & \textbf{Reference} \\
	\hline
	$\chi_B(G) > \aleph_0$ & Yes (size 1: $\mathbb{G}_0$) & \cite{kechris1999borel} \\
	$\chi_B(G) > 2$ & Yes (size 1: $\mathbb{L}_0$) & \cite{L0paper} \\
	$\chi_B(G) > k$, $k\geq 3$ & No ($\mathbf\Pi^1_2$-complete) & \cite{Todorcevic_2021_ComplexProbBorel} \\
	\hline
	$\vec\chi_B(D) > \aleph_0$ & Yes (size $\mathfrak{c}$) & \cite{RaghavanXiao2024} \\
	$\vec\chi_B(D) > k$, $k\geq 2$ & No ($\mathbf\Pi^1_2$-complete) & Prop.~\ref{prop:main}, Rmk.~\ref{rmk:dichrom-recovers-chrom} \\
	\hline
\end{tabular}
\end{center}
\medskip

\noindent
The picture reveals a notable difference between the undirected and directed settings at the uncountable threshold: a single graph suffices as a basis in the undirected case, while a continuum-size family is needed for directed graphs.
With finitely many colors, both settings exhibit the same completeness obstruction to a countable basis, though at different thresholds: the phenomenon begins at $k=3$ for undirected graphs and already at $k=2$ for directed graphs, mirroring the classical NP-completeness landscape.
The case $k=2$ for directed graphs, established in this paper, is the last remaining piece of this picture (see Remark~\ref{rmk:dichrom-recovers-chrom}).

Several questions are posed in Section~\ref{sec:questions}.

\subsection{Organization}
Section~\ref{sec:framework} reviews Thornton's coding framework.
Section~\ref{sec:classical} recalls the classical reduction from 3-uniform hypergraph 2-coloring to directed graph 2-dicoloring.
Section~\ref{sec:borel} carries out the Borel lifting.
Section~\ref{sec:questions} collects open questions.

\section{Coding framework}\label{sec:framework}

Briefly recall the coding apparatus from~\cite[Appendix~A]{thornton-thesis}; see also \cite[Section~4]{thornton2022algebraic}.
The reader already familiar with this framework may skip to Section~\ref{sec:classical}.

Fix a good $\omega$-parametrization $U\subseteq\omega\times\NN$ of $\Pi^1_1$ \cite[Theorem~A.5]{thornton2022algebraic}.

\begin{definition}[{\cite[Definition~A.6]{thornton2022algebraic}}]\label{def:simple-nice-codes}
	A \emph{simple code} for a $\Delta^1_1$ set~$B\subseteq\NN$ is a pair $(e,i)\in\omega^2$ with $U_e = \NN\setminus U_i = B$.

	A \emph{nice coding} is a triple $(\CC,\DD^\Pi,\DD^\Sigma)$ where:
	\begin{enumerate}[label=(\roman*)]
		\item $\CC\subseteq\omega$ is $\Pi^1_1$ (the set of \emph{codes});
		\item $\DD^\Pi\subseteq\omega\times\NN$ is $\Pi^1_1$ and $\DD^\Sigma\subseteq\omega\times\NN$ is $\Sigma^1_1$;
		\item for every $e\in\CC$, $\DD^\Pi_e = \DD^\Sigma_e$;
		\item every $\Delta^1_1$ set $B\subseteq\NN$ has a code $e\in\CC$ with $B = \DD^\Pi_e$;
		\item there are recursive functions translating between simple and nice codes in both directions.
	\end{enumerate}
\end{definition}

Fix nice codings for $\Delta^1_1(\NN^k)$ for all~$k$, writing $\DD_e$ for $\DD^\Pi_e$.

\begin{lemma}[{\cite[Lemma~A.7]{thornton2022algebraic}}]\label{lemma:toolkit}
	Fix a $\Delta^1_1$ linear order~$\preceq$ on~$\NN$.
	The following operations are $\Delta^1_1$ in the codes:
	\begin{enumerate}[label=(\arabic*)]
		\item $(A,B)\mapsto A\cup B$;
		\item $(A,B)\mapsto A\cap B$;
		\item $(A,B)\mapsto A\times B$;
		\item $A\mapsto\dom(A)$, for relations with countable sections;
		\item $(A,f)\mapsto f(A)$, for countable-to-one~$f$;
		\item $A\mapsto$ the enumeration function of the finite sections of~$A$ along~$\preceq$.
	\end{enumerate}
\end{lemma}

This lemma can typically be used as a black box: to verify that a construction is $\Delta^1_1$ in the codes, it suffices to express it as a composition of the preceding operations.

\section{The classical reduction}\label{sec:classical}

Now, I recall the reduction from 2-coloring of 3-uniform hypergraphs to 2-dicoloring of directed graphs.
The NP-completeness of 2-dicoloring follows by combining this reduction with the well-known NP-completeness of 2-coloring 3-uniform hypergraphs.
The reduction is due to Bokal et al.~\cite{Bokal_2004_CircularChrom}; I present it in a form convenient for the Borel lifting.

\begin{fact}\label{fact:hyp-np}
	The problem of deciding whether a 3-uniform hypergraph admits a 2-coloring is NP-complete.
\end{fact}
 
\begin{definition}[Template gadget]\label{def:gadget}
	Let $$T_{\mathrm{shared}}:=\{a,b,c\}\quad T_{\mathrm{aux}}:=\{a',b',c'\}\quad T:=T_{\mathrm{shared}}\cup T_{\mathrm{aux}}.$$
	The \emph{template gadget}~$F$ is the directed graph on vertex set~$T$ with arc set~$E_T$ as depicted in Figure~\ref{fig:gadget}.
\end{definition}

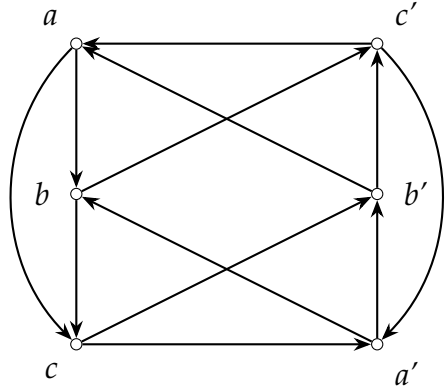
\begin{figure}[ht]
    \centering

\begin{tikzpicture}[
    >=Stealth, 
    node distance=2.5cm,
    dot/.style={circle, draw, fill=white, inner sep=1.5pt},
    every node/.style={font=\large\itshape}
]

    \node[dot] (a) at (0, 4) {};
    \node[dot] (b) at (0, 2) {};
    \node[dot] (c) at (0, 0) {};
    
    \node[dot] (cp) at (4, 4) {};
    \node[dot] (bp) at (4, 2) {};
    \node[dot] (ap) at (4, 0) {};

    \node[above left=2pt of a]  {a};
    \node[left=4pt of b]        {b};
    \node[below left=2pt of c]  {c};
    
    \node[above right=2pt of cp] {c'};
    \node[right=4pt of bp]       {b'};
    \node[below right=2pt of ap] {a'};

    \draw[->, thick] (a) -- (b);
    \draw[->, thick] (b) -- (c);
    \draw[->, thick] (ap) -- (bp);
    \draw[->, thick] (bp) -- (cp);

    \draw[->, thick] (cp) -- (a);
    \draw[->, thick] (c) -- (ap);
    \draw[->, thick] (bp) -- (a);
    \draw[->, thick] (b) -- (cp);
    \draw[->, thick] (ap) -- (b);
    \draw[->, thick] (c) -- (bp);

    \draw[->, thick] (a) to[bend right=45] (c);
    \draw[->, thick] (cp) to[bend left=45] (ap);

\end{tikzpicture}
    \caption{The template gadget~$F$. Shared vertices $\{a,b,c\}$ correspond to hypergraph vertices; auxiliary vertices $\{a',b',c'\}$ are local to each gadget copy.}
    \label{fig:gadget}
\end{figure}

\begin{lemma}[Gadget property]\label{lemma:gadget-property}
	A map $\sigma\colon\{a,b,c\}\to 2$ extends to a 2-dicoloring of~$F$ if and only if~$\sigma$ is not constant (i.e., not all three shared vertices receive the same color).
\end{lemma}

\begin{proof}
	This is a finite verification; see~\cite[Theorem~3.1]{Bokal_2004_CircularChrom}.
\end{proof}

\begin{lemma}[Classical reduction]\label{lem:classical}
For every 3-uniform hypergraph~$H$, there is a directed graph~$D(H)$ such that~$H$ is 2-colorable if and only if $D(H)$~is 2-dicolorable.
\end{lemma}

\begin{proof}
Let $V = V(H)$ and fix a linear order on~$V$.
Define $$E_{\mathrm{sort}} = \bigl\{(v_1,v_2,v_3)\in E(H) : v_1 < v_2 < v_3\bigr\},$$
a canonical listing of hyperedges as ordered triples.
For each $$e=(v_1,v_2,v_3)\in E_{\mathrm{sort}},$$ place a copy~$F_e$ of the template gadget, identifying the shared labels $a,b,c$ with the hypergraph vertices $v_1,v_2,v_3$ and introducing fresh auxiliary vertices $(e,a'),(e,b'),(e,c')$.
The directed graph~$D(H)$ has vertex set $$V_D = V \cup \bigl\{(e,t) : e\in E_{\mathrm{sort}},\; t\in T_{\mathrm{aux}}\bigr\}$$
and arcs given by the union of the arcs of all copies~$F_e$, with each shared label mapped to the corresponding vertex of~$V$.

First, suppose that $c\colon V\to 2$ is a proper 2-coloring of the hypergraph~$H$.
For each~$e$, the induced coloring on the shared vertices of~$F_e$ is non-constant, so by Lemma~\ref{lemma:gadget-property} it extends to a 2-dicoloring of~$F_e$.
Since auxiliary vertices of distinct gadgets are disjoint, combining any choice of extensions with~$c$ on shared vertices gives a global 2-dicoloring of~$D(H)$.

Conversely, suppose that $\psi\colon V_D\to 2$ be a 2-dicoloring of~$D(H)$ and set $\varphi = \psi\restriction_V$.
For any $e = (v_1,v_2,v_3)\in E_{\mathrm{sort}}$, the restriction $\psi\restriction_{V(F_e)}$ is a 2-dicoloring of~$F_e$, so by the contrapositive of Lemma~\ref{lemma:gadget-property}, $\varphi(v_1),\varphi(v_2),\varphi(v_3)$ are not all equal.
Hence $\varphi$ is a 2-coloring of~$H$.
\end{proof}

\begin{remark}\label{rmk:reverse-simple}
The converse direction requires no selection or choice: $\varphi$ is simply the restriction of~$\psi$ to the original vertex set~$V$, which is a subset of~$V_D$ by construction.
This is be important in Section~\ref{sec:borel}, where it means that no uniformization argument is needed for the reverse direction of the Borel reduction.
\end{remark}

\section{The Borel reduction}\label{sec:borel}

I now verify that the construction of Section~\ref{sec:classical} lifts to the Borel setting.
The argument follows the template of~\cite[Theorem~A.9]{thornton2022algebraic}, where a similar verification for edge 3-coloring is carried out.

\begin{fact}[{\cite[Corollary~4.6]{thornton2022algebraic}}]\label{fact:hyp-borel}
The set of nice codes for locally finite Borel 3-uniform hypergraphs admitting a Borel 2-coloring is $\mathbf\Sigma^1_2$-complete.
\end{fact}

It therefore suffices to produce a $\Delta^1_1$ reduction from the set of codes for Borel 2-colorable 3-uniform hypergraphs to the set of codes for Borel 2-dicolorable directed graphs.

\subsection{The construction is \texorpdfstring{$\Delta^1_1$}{delta-one-one} in the codes}

Let $H$ be a locally finite Borel 3-uniform hypergraph with vertex set $V\subseteq\NN$ and hyperedge set $E\subseteq V^3$.
We fix a $\Delta^1_1$ linear order~$\preceq$ on~$\NN$ and encode the template gadget~$F$, its label set~$T$, and its arc set~$E_T$ by canonical finite sequences in~$\NN$.

Begin by sorting the hyperedges.
$$E_{\mathrm{sort}} = \{(a,b,c)\in E : a\preceq b\preceq c\}.$$
This is obtained from~$E$ by intersection with the graph of~$\preceq$, hence is $\Delta^1_1$ on codes by Lemma~\ref{lemma:toolkit}(2).

Then, build the vertex set.
Define $V_D = V \cup \bigl(E_{\mathrm{sort}}\times T_{\mathrm{aux}}\bigr)$, which is a union of~$V$ with a product, hence $\Delta^1_1$ on codes by Lemma~\ref{lemma:toolkit}(1),(3).

Finally, build the arc set.
For each $e = (v_1,v_2,v_3)\in E_{\mathrm{sort}}$ and each template arc $(u,v)\in E_T$, define the \emph{substitution map} $\varphi_H\colon E_{\mathrm{sort}}\times E_T\to V_D^2$ that sends a pair $(e,(u,v))$ to the corresponding arc of~$D(H)$, by replacing each label $\ell\in T$ with its realization in~$D(H)$:
\begin{itemize}[nosep]
	\item if $\ell\in T_{\mathrm{shared}}$ is the shared label in position~$i$ (i.e., $\ell = a,b,c$ for $i=1,2,3$ respectively), it maps to the vertex $v_i\in V$ occurring in $e=(v_1,v_2,v_3)$;
	\item if $\ell\in T_{\mathrm{aux}}$, it maps to $(e,\ell)\in E_{\mathrm{sort}}\times T_{\mathrm{aux}}$.
\end{itemize}
This substitution is a recursive operation on the components of the tuple $(e,(u,v))$: it extracts coordinates of~$e$ (a projection) or forms pairs with~$e$ (a product), both of which are $\Delta^1_1$ operations.
Hence $\varphi_H$ is $\Delta^1_1$ in the code of~$H$ and in the (fixed, computable) code of~$E_T$.

The arc set is $A_D = \varphi_H(E_{\mathrm{sort}}\times E_T)$.
To apply Lemma~\ref{lemma:toolkit}(5), we need $\varphi_H$ to be countable-to-one.
This holds: for any arc $\alpha\in A_D$, the preimage $\varphi_H^{-1}(\alpha)$ is finite, since $E_T$ is finite and, by local finiteness of~$H$, each vertex of~$V$ appears in only finitely many hyperedges in~$E_{\mathrm{sort}}$.
By Lemma~\ref{lemma:toolkit}(1),(3),(5), the arc set~$A_D$ is $\Delta^1_1$ on codes.

Applying the simple-to-nice code translation from the coding framework (Definition~\ref{def:simple-nice-codes}(v)), we obtain:

\begin{lemma}\label{lem:code-reduction}
	The function that maps the code of $H$ to the code of $D(H)$ is $\Delta^1_1$.
\end{lemma}

\subsection{Borel colorability is preserved}

\begin{claim}\label{claim:forward}
	If $H$ is Borel 2-colorable, then $D(H)$ is Borel 2-dicolorable.
\end{claim}

\begin{proof}
Let $c\colon V\to 2$ be a Borel 2-coloring of~$H$.
For each $e\in E_{\mathrm{sort}}$, the restriction of~$c$ to the shared vertices of~$F_e$ is non-constant, so by the gadget property (Lemma~\ref{lemma:gadget-property}) there exists at least one extension to a 2-dicoloring of~$F_e$.
Denote by $\CF$ the finite set of maps $T\to 2$ and define $S$ as the set of all $(e,\sigma)\in E_{\mathrm{sort}}\times\CF$ such that $\sigma$ is a 2-dicoloring of $F_e$ extending the shared-vertex colors.
The set~$S$ is Borel since the defining condition involves only finitely many acyclicity checks and color-matching conditions, all determined by the Borel coloring~$c$ and the fixed finite template.
Furthermore, $S$ has finite nonempty sections~$S_e$.
By the Luzin--Novikov uniformization theorem, there is a Borel selector $e\mapsto\sigma(e)\in S_e$.

Since auxiliary vertices of distinct gadgets are disjoint, the map $\psi\colon V_D\to 2$ defined by
$$\psi(v) =\begin{cases}c(v),&v\in V;\\
	\sigma(e)(t),&(e,t)\in E_{\mathrm{sort}}\times T_{\mathrm{aux}}\end{cases}$$
is a well-defined Borel 2-dicoloring of~$D(H)$.
\end{proof}

\begin{claim}\label{claim:reverse}
If $D(H)$ is Borel 2-dicolorable, then $H$ is Borel 2-colorable.
\end{claim}

\begin{proof}
Let $\psi\colon V_D\to 2$ be a Borel 2-dicoloring of~$D(H)$ and define $\varphi = \psi\restriction_V$.
Since $V\subseteq V_D$ is Borel, $\varphi$~is Borel.
For any hyperedge $e=(v_1,v_2,v_3)\in E_{\mathrm{sort}}$, the gadget~$F_e$ is a directed subgraph of~$D(H)$, so $\psi\restriction_{V(F_e)}$ is a 2-dicoloring of~$F_e$, and the gadget property implies that $\varphi(v_1),\varphi(v_2),\varphi(v_3)$ are not all equal.
Hence~$\varphi$ is a Borel 2-coloring of~$H$.
\end{proof}

\begin{remark}\label{rmk:no-lf-reverse}
Note that the reverse direction (Claim~\ref{claim:reverse}) uses neither the local finiteness of~$H$ nor any uniformization.
Local finiteness is needed only for two purposes: it is a hypothesis of Fact~\ref{fact:hyp-borel}, and it guarantees that the sections~$S_e$ in the forward direction are finite (enabling the application of Luzin--Novikov).
\end{remark}

\subsection{Proof of the main results}

\begin{proof}[Proof of Proposition~\ref{prop:main}]
By Fact~\ref{fact:hyp-borel}, the set of nice codes for locally finite Borel 3-uniform hypergraphs with Borel 2-colorings is $\mathbf\Sigma^1_2$-complete.
By Lemma~\ref{lem:code-reduction}, the map $\mathrm{code}(H)\mapsto\mathrm{code}(D(H))$ is~$\Delta^1_1$.
By Claims~\ref{claim:forward} and~\ref{claim:reverse}, $H$~is Borel 2-colorable if and only if $D(H)$~is Borel 2-dicolorable.
Thus the set of nice codes for Borel 2-dicolorable directed graphs is $\mathbf\Sigma^1_2$-hard.

For the upper bound, note that a directed graph~$D$ (given by a nice code) admits a Borel 2-dicoloring if and only if there exists a nice code~$e_c$ such that $\DD_{e_c}$ defines a valid 2-dicoloring of~$D$.
The condition ``$\DD_{e_c}$ is a 2-dicoloring'' is $\Pi^1_1$ in the codes $e_c$ and $e_D$, as it requires totality and that each color class induces no directed cycle, both of which are universal conditions.
The existential quantification over~$e_c$ gives a $\Sigma^1_2$ condition.
Hence the set of codes for Borel 2-dicolorable directed graphs is $\mathbf\Sigma^1_2$-complete, and its complement, which are codes for directed graphs with $\vec\chi_B(D)\geq 3$, is $\mathbf\Pi^1_2$-complete.
\end{proof}

\begin{remark}\label{rmk:sigma-pi-convention}
The convention used in~\cite{Todorcevic_2021_ComplexProbBorel} and~\cite{thornton2022algebraic} is to state the $\mathbf\Sigma^1_2$-completeness of the class of graphs \emph{admitting} a coloring.
We follow this convention.
\end{remark}

\begin{proof}[Proof of Corollary~\ref{cor:no-basis}]
Suppose a countable basis $\mathbf{B} = \{H_1,H_2,\dots\}$ exists, so that
$$
	\vec\chi_B(D) \geq 3
	\quad\Longleftrightarrow\quad
	\exists\, H_n\in\mathbf{B},\ \exists\text{ Borel homomorphism } H_n\to D.
$$
For each fixed~$H_n$, the condition ``there exists a Borel homomorphism $H_n\to D$'' is $\mathbf\Sigma^1_2$ in the code for~$D$: one existential quantifier over a nice code~$e_f$ for a Borel function $f\colon V(H_n)\to V(D)$, and the requirement that~$f$ preserves arcs is $\Pi^1_1$ in the codes.

The countable union $\bigcup_n \{D : H_n\to_B D\}$ is therefore~$\mathbf\Sigma^1_2$.
And Proposition~\ref{prop:main} asserts that the $\{D : \vec\chi_B(D)\geq 3\}$ is $\mathbf\Pi^1_2$-complete.
A set that is simultaneously $\mathbf\Sigma^1_2$ and $\mathbf\Pi^1_2$-hard would give $\mathbf\Pi^1_2 \subseteq \mathbf\Sigma^1_2$, contradicting the fact that the projective hierarchy is strictly ordered.
\end{proof}

\begin{remark}\label{rmk:complexity-lower-bound}
The argument proves more than the non-existence of a countable basis.
If $\mathbf{B}$ is any basis for the class of Borel directed graphs with $\vec\chi_B(D)\geq 3$, then the right-hand side of
$$
	\vec\chi_B(D)\geq 3 \quad\Longleftrightarrow\quad \exists\, H\in\mathbf{B},\ \exists\text{ Borel homomorphism }H\to D
$$
must define a $\mathbf\Pi^1_2$-complete set, since the left-hand side is $\mathbf\Pi^1_2$-complete by Proposition~\ref{prop:main}.
If $\mathbf{B}\in\mathbf\Sigma^1_2$ (viewed as a subset of the code space), then ``$H\in\mathbf{B}$'' is a $\mathbf\Sigma^1_2$ condition on the code of~$H$, and ``there exists a Borel homomorphism $H\to D$'' is $\mathbf\Sigma^1_2$ in the codes of~$H$ and~$D$ (as in the proof of Corollary~\ref{cor:no-basis}).
The conjunction of two $\mathbf\Sigma^1_2$ conditions is $\mathbf\Sigma^1_2$, and existential quantification preserves $\mathbf\Sigma^1_2$, so the right-hand side would be $\mathbf\Sigma^1_2$ in the code of~$D$.
But this contradicts $\mathbf\Pi^1_2$-completeness of the left-hand side, since $\mathbf\Sigma^1_2\neq\mathbf\Pi^1_2$.
Hence $\mathbf{B}\notin\mathbf\Sigma^1_2$: any basis must be at least $\mathbf\Pi^1_2$-hard as a subset of the code space, and is therefore at least as complex to describe as the set of all directed graphs without a Borel acyclic 2-coloring.
The separation $\mathbf\Sigma^1_2\neq\mathbf\Pi^1_2$ used here follows from $\Pi^1_1$-determinacy, which is provable from standard large cardinal hypotheses.
\end{remark}

\section{Questions}\label{sec:questions}
 
This section collects several natural questions suggested by our results, the work of Raghavan and Xiao~\cite{RaghavanXiao2024}, and suggestions of Thornton.
 
\begin{enumerate}[label=\textbf{Q\arabic*.},leftmargin=*]
	\item \textbf{Minimality of the Raghavan--Xiao basis.}
	Raghavan and Xiao~\cite{RaghavanXiao2024} produce a continuum-size family of directed graphs that serves as a basis for Borel directed graphs with uncountable Borel dichromatic number.
	This family consists of pairwise Borel-incomparable directed graphs.
	\emph{Does every basis for this class have size~$\mathfrak{c}$?}
	Raghavan has noted (personal communication) that the directed graphs in his family are pairwise non-homomorphic, but it is unknown whether a smaller basis exists.
	
	\item \textbf{Structural restrictions.}
	The result presented here applies to the full class of Borel directed graphs.
	\emph{Does a countable basis exist when one restricts to structurally simpler classes, say locally finite acyclic Borel directed graphs, or directed graphs generated by a single Borel function?}
	In the undirected case, Miller~\cite{Miller_2008_MeasurableChromaticNumbers} proved that there is no countable $\leq_B$-basis for graphs of the form~$G_f$ (where $f$ is a Borel function) with $\chi_B(G_f)\geq 3$.
	The analogous question for directed graphs generated by a single Borel function and dichromatic number remains open.
	
	\item \textbf{Borel dichromatic number and Borel order dimension.}
	Raghavan and Xiao~\cite{RaghavanXiao2024} established a close connection between Borel dichromatic number and a notion of Borel order dimension for Borel quasi-orders.
	\emph{Does the complexity result of Proposition~\ref{prop:main} yield new complexity results for Borel order dimension?}
	
	\item \textbf{Dichotomies for complete multipartite graphs.}
	It is known that checking whether a Borel graph in which every connected component is complete $k$-partite admits a Borel $k$-coloring is~$\Pi^1_1$, but no satisfying dichotomy theorem is known for this problem.
	Thornton has suggested (personal communication) that a finite basis may exist for each~$k$, though its size appears to grow faster than~$n!$.
	\emph{Is there an explicit finite basis for each~$k$?  What is its growth rate?}
	
	\item \textbf{Bounded width CSPs.}
	Among the constraint satisfaction problems studied in the Borel setting, the \emph{bounded width} problems occupy a distinguished position.
	Thornton~\cite{thornton2022algebraic} showed that the structures whose every solvable Borel instance has a Borel solution are exactly the width~$1$ structures, and proved partial complexity results for certain bounded width structures.
	Recent work of Grebík and Vidnyánszky~\cite{GV2025} shows that the split between easy and hard problems lies at a different place in the Borel setting than in the classical CSP dichotomy: in particular, there is no dichotomy for any Borel CSP that is not of bounded width.
	However, sharp complexity bounds within bounded width remain known only in very special cases.
	\emph{Can the exact complexity dividing line within bounded width CSPs be determined?}
	This problem appears to be very difficult; even identifying good illustrative examples of bounded width CSPs beyond the well-studied special cases remains open.
	
	\item \textbf{Measurable arboricity.}
	In the undirected setting, the \emph{arboricity} of a graph (the minimum number of acyclic sets needed to cover the vertex set) is the undirected analogue of the dichromatic number.
	Several basic questions about arboricity in measurable combinatorics remain open.
	For instance:
	\emph{Can the $\mu$-measurable arboricity of a probability-measure-preserving graph differ from its (classical) arboricity by more than~$1$?}
	And: \emph{What is the $\mu$-measurable arboricity of the $k$-th power of the Bernoulli shift graph of~$F_2$?}
\end{enumerate}

\subsection*{Acknowledgments}

I thank Dilip Raghavan for suggesting that the NP-completeness proof for dichromatic number could be relevant to the Borel setting, my doctoral advisor Spencer Unger for pointing me towards Thornton's coding framework, and Riley Thornton for confirming that 2-dicoloring is not a CSP in the sense of~\cite{thornton-thesis}, advising that the approach of~\cite[Appendix~A]{thornton2022algebraic} should be followed instead, and suggesting several of the questions in Section~\ref{sec:questions}.

\nocite{*}
\bibliographystyle{alphaurl}
\bibliography{bibliography/references}

\end{document}